\documentclass[12pt, reqno]{amsart}
\usepackage[T1]{fontenc}
\usepackage{amsfonts}
\usepackage{amsmath}
\usepackage{amsthm}
\usepackage{amssymb}
\usepackage{geometry}
\usepackage{graphicx}
\usepackage{xcolor}
\usepackage{mathtools}
\usepackage[colorlinks=true, linkcolor=blue]{hyperref}
\usepackage{cleveref}
\usepackage[english]{babel}
\usepackage{lineno}
\usepackage{float}

\newtheorem{theorem}{Theorem}[section]

\newtheorem{lemma}[theorem]{Lemma}
\newtheorem{conjecture}[theorem]{Conjecture}
\newtheorem{proposition}[theorem]{Proposition}

\newtheorem{claim}[theorem]{Claim}

\usepackage{tikz}
\usetikzlibrary{positioning}
\usetikzlibrary{decorations,arrows}
\usetikzlibrary{decorations.markings}
\numberwithin{equation}{section}

\usepackage{rustic}
\usepackage[T1]{fontenc}

\emergencystretch=\maxdimen
\title{Strong edge-coloring of sparse graphs with Ore-degree 7 or 8}
\author{Runze Wang}
\address[]{Department of Mathematical Sciences, University of Memphis, Memphis, TN 38152, USA}
\email{runze.w@hotmail.com}
\thanks{}
\date{\today}

\begin{document}

\sloppy

\begin{abstract}
    In a strong edge-coloring of a graph $G=(V,E)$, any two edges of distance at most $2$ get distinct colors. The strong chromatic index of $G$, denoted by $\chi_s'(G)$, is the minimum number of colors needed in a strong edge-coloring of $G$. The Ore-degree of $G$ is defined by $\max\{d(u)+d(v):uv\in E\}$. In this paper, we apply the discharging method and make use of Hall's marriage theorem to prove two results toward a conjecture by Chen et al. First, we prove that if $G$ is a graph with Ore-degree $7$ and maximum average degree less than $\frac{34}{11}$, then $\chi_s'(G)\le 13$. This result improves the previous best bound from $\frac{40}{13}$ to $\frac{34}{11}$. Second, we prove that if $G$ is a graph with Ore-degree $8$ and maximum average degree less than $\frac{113}{31}$, then $\chi_s'(G)\le 20$.
\end{abstract}

\maketitle

\section{Introduction}

In this paper, we only study finite simple connected graphs.

In a graph $G=(V,E)$, the distance between two edges $e$ and $e'$ is the distance between the corresponding vertices in the line graph of $G$. Thus, if $e$ and $e'$ are incident, then they have distance $1$; if $e$ and $e'$ are not incident, but are joined by another edge, then they have distance $2$. For convenience, if the distance between $e$ and $e'$ is $1$ or $2$, then we say that $e$ \emph{sees} $e'$. (We do not say $e$ sees itself, because the distance between $e$ and itself is $0$.)

The concept of \emph{strong edge-coloring} of graphs was introduced by Fouquet and Jolivet \cite{FJ} in 1983. In a strong edge-coloring of a graph $G$, if two edges see each other, then they get distinct colors. If we can use a $k$-color assignment $f:E(G)\longrightarrow \{1,2,\dots,k\}$ to give $G$ a strong edge-coloring, then $f$ is called a \emph{strong $k$-edge-coloring} of $G$. The \emph{strong chromatic index} of $G$, denoted by $\chi_s'(G)$, is the minimum number of colors needed in a strong edge-coloring of $G$.

Let $\Delta$ be the maximum degree of $G$. Using a probabilistic method, Molloy and Reed \cite{MR} proved that $\chi_s'(G)\le 1.998\Delta^2$ for large $\Delta$. This result was improved to $\chi_s'(G)\le 1.93\Delta^2$ by Bruhn and Joos \cite{BJ}, and then to $\chi_s'(G)\le 1.835\Delta^2$ by Bonamy et al.~\cite{BPP}. The best bound so far is $\chi_s'(G)\le 1.772\Delta^2$, given by Hurley et al.~\cite{HDK}.

There is a famous conjecture made by Erd\H os and Ne\v set\v ril \cite{EN}.

\begin{conjecture}[Erd\H os and Ne\v set\v ril \cite{EN}]\label{conj1}
Let $G$ be a graph with maximum degree $\Delta$. Then
\[
    \chi_s'(G)\le\begin{cases}
        \frac{5}{4}\Delta^2 &if\ \Delta\ is\ even, \\
        \frac{5}{4}\Delta^2-\frac{\Delta}{2}+\frac{1}{4} &if\ \Delta\ is\ odd.
    \end{cases}
\]
\end{conjecture}

For $\Delta\le 2$, this conjecture is trivially true. The first non-trivial case $\Delta=3$ was verified by Andersen \cite{An} and independently by Hor\'ak et al.~\cite{HHT}. For $\Delta=4$, Huang et al.~\cite{HSY} proved that $\chi_s'(G)\le 21$ (the conjectured bound is $20$). For $\Delta=5$, Zang \cite{Za} proved that $\chi_s'(G)\le 37$ (the conjectured bound is $29$).

The \emph{maximum average degree} of $G$, denoted by $mad(G)$, is defined by
\begin{align*}
    mad(G):=\max\bigl\{\bar{d}(G'):G'\subseteq G\bigr\},
\end{align*}
where $\bar{d}(G')$ is the average degree of $G'$. We use $mad(G)$ to measure the sparsity of $G$.

For sparse graphs, Lv et al.~\cite{LLY} proved that if $G$ is a graph with $\Delta=4$ and $mad(G)<\frac{51}{13}$, then Conjecture \ref{conj1} is true for $G$; Lu et al.~\cite{LLH} proved that if $G$ is a graph with $\Delta=5$ and $mad(G)<\frac{22}{5}$, then Conjecture \ref{conj1} is true for $G$.

The \emph{Ore-degree} of $G=(V,E)$, denoted by $\theta(G)$, is defined by
\begin{align*}
    \theta(G):=\max\{d(u)+d(v):uv\in E\}.
\end{align*}
In some papers, e.g.~\cite{CCZZ,LL,Wa}, the Ore-degree of $G$ is also called "the edge weight of $G$" or "the maximum edge weight of $G$".

Chen et al.~\cite{CHYZ} made the following conjecture, which is similar to Conjecture \ref{conj1} but uses Ore-degree $\theta(G)$ instead of maximum degree $\Delta$.

\begin{conjecture}[Chen et al.~\cite{CHYZ}]\label{conj2}
    Let $G$ be a graph with Ore-degree $\theta(G)\ge 5$. Then
    \[\chi_s'(G)\le \begin{cases}
        5\bigl\lceil\frac{\theta(G)}{4}\bigr\rceil^2-8\bigl\lceil\frac{\theta(G)}{4}\bigr\rceil+3 &if\ \theta(G)\equiv 1\mod 4; \\
        5\bigl\lceil\frac{\theta(G)}{4}\bigr\rceil^2-6\bigl\lceil\frac{\theta(G)}{4}\bigr\rceil+2 &if\ \theta(G)\equiv 2\mod 4; \\
        5\bigl\lceil\frac{\theta(G)}{4}\bigr\rceil^2-4\bigl\lceil\frac{\theta(G)}{4}\bigr\rceil+1 &if\ \theta(G)\equiv 3\mod 4; \\
        5\bigl\lceil\frac{\theta(G)}{4}\bigr\rceil^2 &if\ \theta(G)\equiv 0\mod 4.
    \end{cases}
    \]
\end{conjecture}

The case $\theta(G)=5$ was verified by Wu and Lin \cite{WL}. The case $\theta(G)=6$ was verified by Nakprasit and Nakprasit \cite{NN}, and independently by Chen et al.~\cite{CHYZ}. 

For $\theta(G)=7$, the conjectured bound is $\chi_s'(G)\le 13$, and the following results have been proved:
\begin{itemize}
    \item Chen et al.~\cite{CHYZ} proved that $\chi_s'(G)\le 15$.
    \item Wang \cite{Wa} proved that $\chi_s'(G)\le 13$ if $mad(G)<\frac{40}{13}$.
    \item Nelson and Yu \cite{NY} proved that $\chi_s'(G)\le 13$ if $G$ is a planar graph.
    \item Lin and Lin \cite{LL} proved that $\chi_s'(G)\le 9$ if $G$ is claw-free.
\end{itemize}

For $\theta(G)=8$, the conjectured bound is $\chi_s'(G)\le 20$, and Chen et al.~\cite{CCZZ} proved that $\chi_s'(G)\le 21$.

In this paper, we prove two major results. 

First, in Section 2, we study the case $\theta(G)=7$ and improve the result in \cite{Wa} from $mad(G)<\frac{40}{13}=3+\frac{1}{13}$ to $mad(G)<\frac{34}{11}=3+\frac{1}{11}$.

\begin{theorem}\label{thm1}
    Let $G$ be a graph with $\theta(G)\le 7$. If $mad(G)<\frac{34}{11}$, then $\chi_s'(G)\le 13$.
\end{theorem}

Then, in Section 3, we study the case $\theta(G)=8$ and verify Conjecture \ref{conj2} for graphs with $mad(G)<\frac{113}{31}=3+\frac{20}{31}$.

\begin{theorem}\label{thm2}
    Let $G$ be a graph with $\theta(G)\le 8$. If $mad(G)<\frac{113}{31}$, then $\chi_s'(G)\le 20$.
\end{theorem}

In the proofs, we need to use Hall's marriage theorem multiple times.
\begin{theorem}[Hall \cite{Ha}]\label{hall}
    Let $S$ be a set, let $n$ be a positive integer, and let $S_1,S_2,\dots,S_n$ be $n$ subsets of $S$. We can find a system of distinct representatives of $\{S_1,S_2,\dots,S_n\}$ if and only if for any $k\in [1,n]$ and any $1\le i_1\neq i_2\neq\dots\neq i_k\le n$, we have $|\bigcup_{j=1}^k S_{i_j}|\ge k$.
\end{theorem}

\section{Proof of Theorem \ref{thm1}}

It is clear that, for a graph $G$ and its subgraph $G'$, we have $mad(G')\le mad(G)$ and $\theta(G')\le \theta(G)$. For a vertex $v\in V(G)$, let $G-v$ denote the subgraph of $G$ obtained by deleting $v$ and all the edges incident to $v$ from $G$. 

Suppose to the contrary that Theorem \ref{thm1} is not true, and $H$ is a minimal counterexample to Theorem \ref{thm1}, which means $\theta(H)\le 7$, $mad(H)<\frac{34}{11}$, and $\chi_s'(H)>13$, but for any $v\in V(H)$, we have $\chi_s'(H-v)\le 13$.

A strong $13$-edge-coloring $f:E(H-v)\longrightarrow \{1,2,\dots,13\}$ of $H-v$ is called a \emph{partial coloring} of $H$. For an edge $e\in E(H)\setminus E(H-v)$, if a color $c\in \{1,2,\dots,13\}$ has not been used on an edge in $H-v$ that $e$ sees, then on the basis of $f$, we can color $e$ by $c$. Let $L_f(e)\subseteq \{1,2,\dots,13\}$ be the set of colors that have not been used on an edge that $e$ sees, and let $l_f(e):=|L_f(e)|$. It is clear that if $e$ sees $k\le 12$ edges in $H-v$, then $l_f(e)\ge 13-k\ge 1$.

The following lemma is obvious. A formal proof can be found in \cite{Wa}.

\begin{lemma}\label{lemma}
    Let $v$ be a vertex in $H$ with $k$ edges $e_1,e_2,\dots,e_k$ incident to it. Let $f:E(H-v)\longrightarrow \{1,2,\dots,13\}$ be a strong $13$-edge-coloring of $H-v$, which is a partial coloring of $H$. If, up to re-ordering $e_1,e_2,\dots,e_k$, we have 
    \begin{align*}
        l_f(e_i)\ge i
    \end{align*}
    for each $i\in [1,k]$, then $f$ can be extended to a strong $13$-edge-coloring of $H$.
\end{lemma}

If the degree of $v\in V(H)$ is $d$, then $v$ is called a \emph{$d$-vertex}. The following claims are proved in \cite{Wa}.

\begin{claim}[Wang \cite{Wa}]
    There are no $1$-vertices, $5$-vertices, or $6$-vertices in $H$.
\end{claim}

This claim tells us that only $2$-vertices, $3$-vertices, and $4$-vertices can be in $H$.

\begin{claim}[Wang \cite{Wa}]\label{2v}
    A $2$-vertex in $H$ is adjacent to two $4$-vertices.
\end{claim}

This claim also tells us that a $3$-vertex in $H$ can only be adjacent to $3$-vertices and $4$-vertices. Thus, we can divide the $3$-vertices into four classes in the following way.

\begin{itemize}
    \item A $3$-vertex is called a $3(A)$-vertex if it is adjacent to three $4$-vertices;
    \item A $3$-vertex is called a $3(B)$-vertex if it is adjacent to two $4$-vertices and one $3$-vertex;
    \item A $3$-vertex is called a $3(C)$-vertex if it is adjacent to one $4$-vertex and two $3$-vertices;
    \item A $3$-vertex is called a $3(D)$-vertex if it is adjacent to three $3$-vertices.
\end{itemize}

\begin{claim}[Wang \cite{Wa}]\label{3dv}
    If a $3(D)$-vertex in $H$ is adjacent to another $3(D)$-vertex, then it is adjacent to two $3(B)$-vertices.
\end{claim}

This claim also implies that a $3(D)$-vertex in $H$ is adjacent to at most one $3(D)$-vertex.

\begin{claim}[Wang \cite{Wa}]\label{4v}
    A $4$-vertex in $H$ is adjacent to at most one $2$-vertex.
\end{claim}

Now, we use Lemma \ref{lemma} to prove some new claims revealing more structural information about $H$.

\begin{claim}\label{triangle}
    A triangle cannot be a subgraph of $H$.
\end{claim}

\begin{proof}
    Suppose to the contrary that $H$ has a triangle subgraph whose vertices are $x_1$, $x_2$, and $x_3$. We know that this triangle does not have a $2$-vertex, because otherwise, by Claim \ref{2v}, the other two vertices in this triangle are both $4$-vertices, but two $4$-vertices cannot be adjacent to each other, as $\theta(H)\le 7$. So, there are two cases.

    \textbf{Case 1.} This triangle has three $3$-vertices.

    Both $x_1 x_2$ and $x_1 x_3$ are edges incident to $x_1$, and we let the third edge incident to $x_1$ be $x_1 y$. By the minimality of $H$, we know that $H-x_1$ has a strong $13$-edge-coloring $f$. But now, it is easy to check that $x_1 y$ sees at most $12$ edges in $H-x_1$, which means $l_f(x_1 y)\ge 1$; $x_1 x_2$ sees at most $9$ edges in $H-x_1$, which means $l_f(x_1 x_2)\ge 4$; and $x_1 x_3$ sees at most $9$ edges in $H-x_1$, which means $l_f(x_1 x_3)\ge 4$. Thus, by Lemma \ref{lemma}, we can extend $f$ to a strong $13$-edge-coloring of $H$, a contradiction.

    \textbf{Case 2.} This triangle has one $4$-vertex and two $3$-vertices.

    Let $x_1$ be the $4$-vertex and let $x_2$ and $x_3$ be the two $3$-vertices. Let $x_2 z$ be the third edge incident to $x_2$ (the other two are $x_2 x_1$ and $x_2 x_3$). By the minimality of $H$, we know that $H-x_2$ has a strong $13$-edge-coloring $f$, which is a partial coloring of $H$. It is easy to check that $x_2 z$ sees at most $13$ edges (including $x_1 x_3$) in $H-x_2$, which is not good because it only implies $l_f(x_2 z)\ge 0$; $x_2 x_1$ sees at most $11$ edges (including $x_1 x_3$) in $H-x_2$, which means $l_f(x_2 x_1)\ge 2$; and $x_2 x_3$ sees at most $10$ edges (including $x_1 x_3$) in $H-x_2$, which means $l_f(x_2 x_3)\ge 3$. 
    
    In $H-x_2$, $x_1 x_3$ sees at most $10$ edges. Now we erase the color on $x_1 x_3$ under $f$ to get a new partial coloring $f'$. Under $f'$, $x_2 z$ sees at most $12$ edges with colors, which means $l_{f'}(x_2 z)\ge 1$; $x_2 x_1$ sees at most $10$ edges with colors, which means $l_{f'}(x_2 x_1)\ge 3$; $x_2 x_3$ sees at most $9$ edges with colors, which means $l_{f'}(x_2 x_3)\ge 4$; and $x_1 x_3$ sees at most $10$ edges with colors, which means $l_{f'}(x_1 x_3)\ge 3$. Thus, by Lemma \ref{lemma}, we can extend $f'$ to a strong $13$-edge-coloring of $H$, a contradiction.
    
    In either case, we get a contradiction, so a triangle cannot be a subgraph of $H$.
\end{proof}

For a cycle with vertices $x_1,x_2,\dots,x_n$ and edges $x_1 x_2,\,x_2 x_3,\,\dots,\,x_{n-1} x_n,\,x_n x_1$, we denote it by $x_1\sim x_2\sim \dots \sim x_n\sim x_1$.

\begin{claim}\label{4cycle}
    A $4$-cycle $x_1\sim x_2\sim x_3\sim x_4\sim x_1$, where $x_1$ is a $3(D)$-vertex, and $x_2$, $x_3$, and $x_4$ are $3$-vertices, cannot be a subgraph of $H$.
\end{claim}

\begin{proof}
    Suppose to the contrary that there is such a $4$-cycle in $H$. Let $x_1 y$ be the third edge incident to $x_1$ (the other two are $x_1 x_2$ and $x_1 x_4$). Note that $y\neq x_3$, because otherwise we have a triangle forbidden by Claim \ref{triangle}. As $x_1$ is a $3(D)$-vertex, we know that $y$ is a $3$-vertex. By the minimality of $H$, we know that $H-x_1$ has a strong $13$-edge-coloring $f$. But now, it is easy to check that $x_1 y$ sees at most $12$ edges in $H-x_1$, which means $l_f(x_1 y)\ge 1$; $x_1 x_2$ sees at most $10$ edges in $H-x_1$, which means $l_f(x_1 x_2)\ge 3$; and $x_1 x_4$ sees at most $10$ edges in $H-x_1$, which means $l_f(x_1 x_4)\ge 3$. Thus, by Lemma \ref{lemma}, we can extend $f$ to a strong $13$-edge-coloring of $H$, a contradiction.
\end{proof}

\begin{claim}\label{5cycle}
    A $5$-cycle $x_1\sim x_2\sim x_3\sim x_4\sim x_5\sim x_1$, where $x_1$, $x_3$, and $x_4$ are $3(D)$-vertices, and $x_2$ and $x_5$ are $3$-vertices, cannot be a subgraph of $H$.
\end{claim}

\begin{proof}
    Suppose to the contrary that there is such a $5$-cycle in $H$. Let $x_1 y$ be the third edge incident to $x_1$ (the other two are $x_1 x_2$ and $x_1 x_5$). As $x_1$ is a $3(D)$-vertex, we know that $y$ is a $3$-vertex. 
    
    We know that $x_1 y$ does not see $x_3 x_4$, because:
    \begin{itemize}
        \item If $x_1 y$ and $x_3 x_4$ are incident, then $y=x_3$ or $y=x_4$, which means $x_1\sim x_2\sim x_3\sim x_1$ or $x_1\sim x_4\sim x_5\sim x_1$ is a triangle, contradicting Claim \ref{triangle}.
        \item If $x_1 y$ and $x_3 x_4$ are not incident, but are joined by another edge, then $y$ is adjacent to $x_3$ or $x_4$, which means $x_1\sim x_2\sim x_3\sim y\sim x_1$ or $x_1\sim x_5\sim x_4\sim y\sim x_1$ is a $4$-cycle forbidden by Claim \ref{4cycle}, a contradiction.
    \end{itemize}
    
    By the minimality of $H$, we know that $H-x_1$ has a strong $13$-edge-coloring $f$. It is easy to check that $x_1 y$ sees at most $12$ edges (excluding $x_3 x_4$) in $H-x_1$, which means $l_f(x_1 y)\ge 1$; $x_1 x_2$ sees at most $11$ edges (including $x_3 x_4$) in $H-x_1$, which means $l_f(x_1 x_2)\ge 2$; and $x_1 x_5$ sees at most $11$ edges (including $x_3 x_4$) in $H-x_1$, which means $l_f(x_1 x_5)\ge 2$.

    In $H-x_1$, $x_3 x_4$ sees at most $10$ edges. Now we erase the color on $x_3 x_4$ under $f$ to get a new partial coloring $f'$. Under $f'$, $x_1 y$ sees at most $12$ edges with colors, which means $l_{f'}(x_1 y)\ge 1$; $x_1 x_2$ sees at most $10$ edges with colors, which means $l_{f'}(x_1 x_2)\ge 3$; $x_1 x_5$ sees at most $10$ edges with colors, which means $l_{f'}(x_1 x_5)\ge 3$; and $x_3 x_4$ sees at most $10$ edges with colors, which means $l_{f'}(x_3 x_4)\ge 3$.

    As $x_1 y$ does not see $x_3 x_4$, we may color them by the same color. There are two cases.

    \textbf{Case 1.} $L_{f'}(x_1 y)\cap L_{f'}(x_3 x_4)\neq\emptyset$.

    In this case, first, we use one color in $L_{f'}(x_1 y)\cap L_{f'}(x_3 x_4)$ on both $x_1 y$ and $x_3 x_4$. Then, there are at least two colors available for $x_1 x_2$ and at least two colors available for $x_1 x_5$, so we can also find feasible colors for $x_1 x_2$ and $x_1 x_5$, and extend $f'$ to a strong $13$-edge-coloring of $H$, a contradiction.

    \textbf{Case 2.} $L_{f'}(x_1 y)\cap L_{f'}(x_3 x_4)=\emptyset$.

    In this case, we invoke Hall's marriage theorem (Theorem \ref{hall}). Let $S_1=L_{f'}(x_1 y)$, $S_2=L_{f'}(x_1 x_2)$, $S_3=L_{f'}(x_1 x_5)$, and $S_4=L_{f'}(x_3 x_4)$. It is clear that, for any $k\in \{1,2,3\}$ and any $1\le i_1\neq i_2\neq\dots\neq i_k\le 4$, we have $|\bigcup_{j=1}^k S_{i_j}|\ge k$. Then, for $k=4$, we have $|S_1\cup S_2\cup S_3\cup S_4|\ge |S_1\cup S_4|=|L_{f'}(x_1 y)\cup L_{f'}(x_3 x_4)|\ge 4$, as $L_{f'}(x_1 y)$ and $L_{f'}(x_3 x_4)$ are disjoint. Thus, by Hall's marriage theorem, we can find a system of distinct representatives of $\{S_1,S_2,S_3,S_4\}$, which means we can find feasible colors for $x_1 y$, $x_1 x_2$, $x_1 x_5$, and $x_3 x_4$ to extend $f'$ to a strong $13$-edge-coloring of $H$, a contradiction.

    In either case, we get a contradiction, so such a $5$-cycle cannot be a subgraph of $H$.
\end{proof}

\begin{claim}\label{pan}
    The graph shown in Figure \ref{fig1}, where $y$ is a $3(C)$-vertex or $3(D)$-vertex, $x_1$ and $x_4$ are $3(D)$-vertices, and $x_2$, $x_3$, and $x_5$ are $3$-vertices, cannot be a subgraph of $H$.
\end{claim}

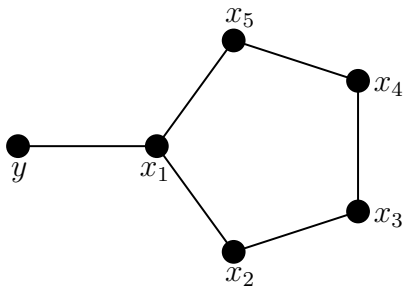
\begin{figure}[H]
    \centering
    \tikzset{every picture/.style={line width=0.75pt}} 

\begin{tikzpicture}[x=0.75pt,y=0.75pt,yscale=-1,xscale=1]

\draw   (208,97.1) -- (246.76,43.75) -- (309.49,64.13) -- (309.49,130.07) -- (246.76,150.45) -- cycle ;
\draw    (138,97.1) -- (208,97.1) ;
\draw  [fill={rgb, 255:red, 0; green, 0; blue, 0 }  ,fill opacity=1 ] (132.5,97.1) .. controls (132.5,94.06) and (134.96,91.6) .. (138,91.6) .. controls (141.04,91.6) and (143.5,94.06) .. (143.5,97.1) .. controls (143.5,100.14) and (141.04,102.6) .. (138,102.6) .. controls (134.96,102.6) and (132.5,100.14) .. (132.5,97.1) -- cycle ;
\draw  [fill={rgb, 255:red, 0; green, 0; blue, 0 }  ,fill opacity=1 ] (202.5,97.1) .. controls (202.5,94.06) and (204.96,91.6) .. (208,91.6) .. controls (211.04,91.6) and (213.5,94.06) .. (213.5,97.1) .. controls (213.5,100.14) and (211.04,102.6) .. (208,102.6) .. controls (204.96,102.6) and (202.5,100.14) .. (202.5,97.1) -- cycle ;
\draw  [fill={rgb, 255:red, 0; green, 0; blue, 0 }  ,fill opacity=1 ] (241.26,43.75) .. controls (241.26,40.71) and (243.73,38.25) .. (246.76,38.25) .. controls (249.8,38.25) and (252.26,40.71) .. (252.26,43.75) .. controls (252.26,46.78) and (249.8,49.25) .. (246.76,49.25) .. controls (243.73,49.25) and (241.26,46.78) .. (241.26,43.75) -- cycle ;
\draw  [fill={rgb, 255:red, 0; green, 0; blue, 0 }  ,fill opacity=1 ] (303.99,64.13) .. controls (303.99,61.09) and (306.45,58.63) .. (309.49,58.63) .. controls (312.52,58.63) and (314.99,61.09) .. (314.99,64.13) .. controls (314.99,67.16) and (312.52,69.63) .. (309.49,69.63) .. controls (306.45,69.63) and (303.99,67.16) .. (303.99,64.13) -- cycle ;
\draw  [fill={rgb, 255:red, 0; green, 0; blue, 0 }  ,fill opacity=1 ] (303.99,130.07) .. controls (303.99,127.04) and (306.45,124.57) .. (309.49,124.57) .. controls (312.52,124.57) and (314.99,127.04) .. (314.99,130.07) .. controls (314.99,133.11) and (312.52,135.57) .. (309.49,135.57) .. controls (306.45,135.57) and (303.99,133.11) .. (303.99,130.07) -- cycle ;
\draw  [fill={rgb, 255:red, 0; green, 0; blue, 0 }  ,fill opacity=1 ] (241.26,150.45) .. controls (241.26,147.42) and (243.73,144.95) .. (246.76,144.95) .. controls (249.8,144.95) and (252.26,147.42) .. (252.26,150.45) .. controls (252.26,153.49) and (249.8,155.95) .. (246.76,155.95) .. controls (243.73,155.95) and (241.26,153.49) .. (241.26,150.45) -- cycle ;

\draw (133,104) node [anchor=north west][inner sep=0.75pt]  [color={rgb, 255:red, 0; green, 0; blue, 0 }  ,opacity=1 ] [align=left] {$\displaystyle y$};
\draw (198,104) node [anchor=north west][inner sep=0.75pt]  [color={rgb, 255:red, 0; green, 0; blue, 0 }  ,opacity=1 ] [align=left] {$\displaystyle x_{1}$};
\draw (241,157) node [anchor=north west][inner sep=0.75pt]  [color={rgb, 255:red, 0; green, 0; blue, 0 }  ,opacity=1 ] [align=left] {$\displaystyle x_{2}$};
\draw (316,127) node [anchor=north west][inner sep=0.75pt]  [color={rgb, 255:red, 0; green, 0; blue, 0 }  ,opacity=1 ] [align=left] {$\displaystyle x_{3}$};
\draw (316,61) node [anchor=north west][inner sep=0.75pt]  [color={rgb, 255:red, 0; green, 0; blue, 0 }  ,opacity=1 ] [align=left] {$\displaystyle x_{4}$};
\draw (241,26) node [anchor=north west][inner sep=0.75pt]  [color={rgb, 255:red, 0; green, 0; blue, 0 }  ,opacity=1 ] [align=left] {$\displaystyle x_{5}$};

\end{tikzpicture}
\caption{A forbidden subgraph of $H$.}
\label{fig1}
\end{figure}

This proof is similar to the proof of Claim \ref{5cycle}.

\begin{proof}
    Suppose to the contrary that $H$ has such a subgraph.
    
    We know that $x_1 y$ does not see $x_3 x_4$, because:
    \begin{itemize}
        \item As shown in Figure \ref{fig1}, $x_1$, $y$, $x_3$, and $x_4$ have been assumed to be different vertices, so $x_1 y$ and $x_3 x_4$ are not incident.
        \item If $x_1 y$ and $x_3 x_4$ are not incident, but are joined by another edge, then $y$ is adjacent to $x_3$ or $x_4$, which means $x_1\sim x_2\sim x_3\sim y\sim x_1$ or $x_1\sim x_5\sim x_4\sim y\sim x_1$ is a $4$-cycle forbidden by Claim \ref{4cycle}, a contradiction.
    \end{itemize}

    By the minimality of $H$, we know that $H-x_1$ has a strong $13$-edge-coloring $f$. It is easy to check that $x_1 y$ sees at most $11$ edges (excluding $x_3 x_4$) in $H-x_1$, which means $l_f(x_1 y)\ge 2$; $x_1 x_2$ sees at most $11$ edges (including $x_3 x_4$) in $H-x_1$, which means $l_f(x_1 x_2)\ge 2$; and $x_1 x_5$ sees at most $11$ edges (including $x_3 x_4$) in $H-x_1$, which means $l_f(x_1 x_5)\ge 2$.

    In $H-x_1$, $x_3 x_4$ sees at most $11$ edges. Now we erase the color on $x_3 x_4$ under $f$ to get a new partial coloring $f'$. Under $f'$, $x_1 y$ sees at most $11$ edges with colors, which means $l_{f'}(x_1 y)\ge 2$; $x_1 x_2$ sees at most $10$ edges with colors, which means $l_{f'}(x_1 x_2)\ge 3$; $x_1 x_5$ sees at most $10$ edges with colors, which means $l_{f'}(x_1 x_5)\ge 3$; and $x_3 x_4$ sees at most $11$ edges with colors, which means $l_{f'}(x_3 x_4)\ge 2$.

    As $x_1 y$ does not see $x_3 x_4$, we may color them by the same color. The same as in Claim \ref{5cycle}, there are two cases.

    \textbf{Case 1.} $L_{f'}(x_1 y)\cap L_{f'}(x_3 x_4)\neq\emptyset$.

    In this case, first, we use one color in $L_{f'}(x_1 y)\cap L_{f'}(x_3 x_4)$ on both $x_1 y$ and $x_3 x_4$. Then, there are at least two colors available for $x_1 x_2$ and at least two colors available for $x_1 x_5$, so we can also find feasible colors for $x_1 x_2$ and $x_1 x_5$, and extend $f'$ to a strong $13$-edge-coloring of $H$, a contradiction.

    \textbf{Case 2.} $L_{f'}(x_1 y)\cap L_{f'}(x_3 x_4)=\emptyset$.

    In this case, we can use Hall's marriage theorem (Theorem \ref{hall}). Let $S_1=L_{f'}(x_1 y)$, $S_2=L_{f'}(x_1 x_2)$, $S_3=L_{f'}(x_1 x_5)$, and $S_4=L_{f'}(x_3 x_4)$. It is clear that, for any $k\in \{1,2,3\}$ and any $1\le i_1\neq i_2\neq\dots\neq i_k\le 4$, we have $|\bigcup_{j=1}^k S_{i_j}|\ge k$. Then, for $k=4$, we have $|S_1\cup S_2\cup S_3\cup S_4|\ge |S_1\cup S_4|=|L_{f'}(x_1 y)\cup L_{f'}(x_3 x_4)|\ge 4$, as $L_{f'}(x_1 y)$ and $L_{f'}(x_3 x_4)$ are disjoint. Thus, we can find a system of distinct representatives of $\{S_1,S_2,S_3,S_4\}$, which means we can find feasible colors for $x_1 y$, $x_1 x_2$, $x_1 x_5$, and $x_3 x_4$ to extend $f'$ to a strong $13$-edge-coloring of $H$, a contradiction.

    In either case, we get a contradiction, so $H$ cannot have a subgraph described in this theorem.
\end{proof}

Now, we further divide the $3(C)$-vertices in $H$ into three subclasses.

\begin{itemize}
    \item A $3(C)$-vertex is called a $3(C_{weak})$-vertex if it is adjacent to two $3(D)$-vertices;
    \item A $3(C)$-vertex is called a $3(C_{strong})$-vertex if it is adjacent to at least one $3(B)$-vertex;
    \item A $3(C)$-vertex is called a $3(C_{moderate})$-vertex if it is neither a $3(C_{weak})$-vertex nor a $3(C_{strong})$-vertex.
\end{itemize}

Since a $3(C_{moderate})$-vertex is not adjacent to $3(B)$-vertices, and it is adjacent to at most one $3(D)$-vertex, we know that a $3(C_{moderate})$-vertex must be adjacent to at least one $3(C)$-vertex.

\begin{claim}\label{3d1}
    A $3(D)$-vertex in $H$ is adjacent to at most one $3(C_{weak})$-vertex.
\end{claim}

\begin{proof}
    Suppose to the contrary that $x$ is a $3(D)$-vertex, and it is adjacent to two $3(C_{weak})$-vertices $y_1$ and $y_2$. Let $z_1$ be the other $3(D)$-vertex adjacent to $y_1$, and let $z_2$ be the other $3(D)$-vertex adjacent to $y_2$. Let $xw$ be the third edge incident to $x$ (the other two are $x y_1$ and $x y_2$). As $x$ is a $3(D)$-vertex, we know that $w$ is a $3$-vertex.

    For convenience, denote $x y_1$ by $e_1$, $x y_2$ by $e_2$, $x w$ by $e_3$, $y_1 z_1$ by $e_4$, and $y_2 z_2$ by $e_5$. 
    
    We know that $e_4\neq e_5$. Because otherwise $z_1=y_2$ and $z_2=y_1$, which is impossible because $z_1$ and $z_2$ are $3(D)$-vertices and $y_1$ and $y_2$ are $3(C_{weak})$-vertices.

    Then, we check that $e_4$ and $e_5$ do not see each other.
    
    If $e_4$ and $e_5$ are incident, then there are three possible cases.

    \begin{enumerate}
        \item $z_1=y_2$. This is impossible because $z_1$ is a $3(D)$-vertex and $y_2$ is a $3(C_{weak})$-vertex.
        \item $z_2=y_1$. This is impossible because $z_2$ is a $3(D)$-vertex and $y_1$ is a $3(C_{weak})$-vertex.
        \item $z_1=z_2$. In this case, $x\sim y_1\sim z_1\sim y_2\sim x$ is a $4$-cycle forbidden by Claim \ref{4cycle}, a contradiction.
    \end{enumerate}

    If $e_4$ and $e_5$ are not incident, but are joined by another edge, then there are four possible cases.

    \begin{enumerate}
        \item $y_1$ and $y_2$ are joined by an edge. In this case, $x\sim y_1\sim y_2\sim x$ is a triangle forbidden by Claim \ref{triangle}, a contradiction.
        \item $y_1$ and $z_2$ are joined by an edge. In this case, $x\sim y_1\sim z_2\sim y_2\sim x$ is a $4$-cycle forbidden by Claim \ref{4cycle}, a contradiction.
        \item $z_1$ and $y_2$ are joined by an edge. In this case, $x\sim y_1\sim z_1\sim y_2\sim x$ is a $4$-cycle forbidden by Claim \ref{4cycle}, a contradiction.
        \item $z_1$ and $z_2$ are joined by an edge. In this case, $x\sim y_1\sim z_1\sim z_2\sim y_2\sim x$ is a $5$-cycle forbidden by Claim \ref{5cycle}, a contradiction.
    \end{enumerate}

    Thus, $e_4$ and $e_5$ do not see each other.

    By the minimality of $H$, we know that $H-x$ has a strong $13$-edge-coloring $f$. It is easy to check that $e_1$ sees at most $11$ edges (including $e_4$ and $e_5$) in $H-x$, which means $l_f(e_1)\ge 2$; $e_2$ sees at most $11$ edges (including $e_4$ and $e_5$) in $H-x$, which means $l_f(e_2)\ge 2$; and $e_3$ sees at most $12$ edges (including $e_4$ and $e_5$) in $H-x$, which means $l_f(e_3)\ge 1$.

    In $H-x$, $e_4$ sees at most $10$ edges, and $e_5$ sees at most $10$ edges. Now we erase the colors on $e_4$ and $e_5$ under $f$ to get a new partial coloring $f'$. Under $f'$, $e_1$ sees at most $9$ edges with colors, which means $l_{f'}(e_1)\ge 4$; $e_2$ sees at most $9$ edges with colors, which means $l_{f'}(e_2)\ge 4$; $e_3$ sees at most $10$ edges with colors, which means $l_{f'}(e_3)\ge 3$; $e_4$ sees at most $10$ edges with colors, which means $l_{f'}(e_4)\ge 3$; and $e_5$ sees at most $10$ edges with colors, which means $l_{f'}(e_5)\ge 3$.

    As $e_4$ does not see $e_5$, we may color them by the same color. There are two cases.

    \textbf{Case 1.} $L_{f'}(e_4)\cap L_{f'}(e_5)\neq\emptyset$.

    In this case, first, we use one color in $L_{f'}(e_4)\cap L_{f'}(e_5)$ on both $e_4$ and $e_5$. Then, there are at least three colors available for $e_1$, at least three colors available for $e_2$, and at least two colors available for $e_3$. So, we can find feasible colors for $e_1$, $e_2$, and $e_3$, and extend $f'$ to a strong $13$-edge-coloring of $H$, a contradiction.

    \textbf{Case 2.} $L_{f'}(e_4)\cap L_{f'}(e_5)=\emptyset$.

    In this case, we invoke Hall's marriage theorem (Theorem \ref{hall}). Let $S_1=L_{f'}(e_1)$, $S_2=L_{f'}(e_2)$, $S_3=L_{f'}(e_3)$, $S_4=L_{f'}(e_4)$, and $S_5=L_{f'}(e_5)$. It is clear that, for any $k\in [1,4]$ and any $1\le i_1\neq i_2\neq\dots\neq i_k\le 5$, we have $|\bigcup_{j=1}^k S_{i_j}|\ge k$. Then, for $k=5$, we have $|S_1\cup S_2\cup S_3\cup S_4\cup S_5|\ge |S_4\cup S_5|=|L_{f'}(e_4)\cup L_{f'}(e_5)|\ge 6>5$, as $L_{f'}(e_4)$ and $L_{f'}(e_5)$ are disjoint. Thus, we can find a system of distinct representatives of $\{S_1,S_2,S_3,S_4,S_5\}$, which means we can find feasible colors for $e_1$, $e_2$, $e_3$, $e_4$, and $e_5$ to extend $f'$ to a strong $13$-edge-coloring of $H$, a contradiction.

    In either case, we get a contradiction, so a $3(D)$-vertex in $H$ is adjacent to at most one $3(C_{weak})$-vertex.
\end{proof}

\begin{claim}\label{3d2}
    Let $x$ be a $3(D)$-vertex in $H$, which is adjacent to three $3(C)$-vertices $y_1$, $y_2$, and $y_3$. If one of $y_1$, $y_2$, and $y_3$ is a $3(C_{weak})$-vertex, then the other two are $3(C_{strong})$-vertices.
\end{claim}

\begin{proof}
    By symmetry, without loss of generality, we may assume that $y_1$ is a $3(C_{weak})$-vertex. Suppose to the contrary that at most one of $y_2$ and $y_3$ is a $3(C_{strong})$-vertex. 
    
    If either one of $y_2$ and $y_3$ is a $3(C_{weak})$-vertex, then $x$ is adjacent to two $3(C_{weak})$-vertices, contradicting Claim \ref{3d1}. Thus, we can assume that $y_2$ is a $3(C_{moderate})$-vertex, and $y_3$ is a $3(C_{moderate})$-vertex or $3(C_{strong})$-vertex.

    We have that $y_1$ is a $3(C_{weak})$-vertex and it is already adjacent to a $3(D)$-vertex $x$. Let $z_1$ be the other $3(D)$-vertex adjacent to $y_1$. We have that $y_2$ is a $3(C_{moderate})$-vertex, so, as mentioned earlier, it is adjacent to a $3(C)$-vertex, and we let this $3(C)$-vertex be $z_2$.

    For convenience, denote $x y_1$ by $e_1$, $x y_2$ by $e_2$, $x y_3$ by $e_3$, $y_1 z_1$ by $e_4$, and $y_2 z_2$ by $e_5$. 
    
    It is clear that $e_4\neq e_5$. Because if $e_4=e_5$, then we have $z_1=y_2$ and $z_2=y_1$. However, $z_1$ is a $3(D)$-vertex, and $y_2$ is a $3(C_{moderate})$-vertex, a contradiction.

    Then, we check that $e_4$ and $e_5$ do not see each other.
    
    If $e_4$ and $e_5$ are incident, then there are three possible cases.

    \begin{enumerate}
        \item $z_1=y_2$. This is impossible because $z_1$ is a $3(D)$-vertex and $y_2$ is a $3(C_{moderate})$-vertex.
        \item $z_2=y_1$. In this case, $x\sim y_1\sim y_2\sim x$ is a triangle forbidden by Claim \ref{triangle}, a contradiction.
        \item $z_1=z_2$. This is impossible because $z_1$ is a $3(D)$-vertex and $z_2$ is a $3(C)$-vertex.
    \end{enumerate}

    If $e_4$ and $e_5$ are not incident, but are joined by another edge, then there are four possible cases.

    \begin{enumerate}
        \item $y_1$ and $y_2$ are joined by an edge. In this case, $x\sim y_1\sim y_2\sim x$ is a triangle forbidden by Claim \ref{triangle}, a contradiction.
        \item $y_1$ and $z_2$ are joined by an edge. In this case, $x\sim y_1\sim z_2\sim y_2\sim x$ is a $4$-cycle forbidden by Claim \ref{4cycle}, a contradiction.
        \item $z_1$ and $y_2$ are joined by an edge. In this case, $x\sim y_1\sim z_1\sim y_2\sim x$ is a $4$-cycle forbidden by Claim \ref{4cycle}, a contradiction.
        \item $z_1$ and $z_2$ are joined by an edge. In this case, $H$ has a subgraph isomorphic to the one forbidden by Claim \ref{pan} ($y_3,\,x,\,y_2,\,z_2,\,z_1,\,y_1$ in this proof correspond respectively to $y,\,x_1,\,x_2,\,x_3,\,x_4,\,x_5$ in Claim \ref{pan}). So we also get a contradiction in this case.
    \end{enumerate}

    Thus, $e_4$ and $e_5$ do not see each other.

    By the minimality of $H$, we know that $H-x$ has a strong $13$-edge-coloring $f$. It is easy to check that, for $i\in \{1,2,3\}$, $e_i$ sees at most $11$ edges (including $e_4$ and $e_5$) in $H-x$, which means $l_f(e_i)\ge 2$.

    In $H-x$, $e_4$ sees at most $10$ edges, and $e_5$ sees at most $11$ edges. Now we erase the colors on $e_4$ and $e_5$ under $f$ to get a new partial coloring $f'$. Under $f'$, $e_1$ sees at most $9$ edges with colors, which means $l_{f'}(e_1)\ge 4$; $e_2$ sees at most $9$ edges with colors, which means $l_{f'}(e_2)\ge 4$; $e_3$ sees at most $9$ edges with colors, which means $l_{f'}(e_3)\ge 4$; $e_4$ sees at most $10$ edges with colors, which means $l_{f'}(e_4)\ge 3$; and $e_5$ sees at most $11$ edges with colors, which means $l_{f'}(e_5)\ge 2$.

    As $e_4$ does not see $e_5$, we may color them by the same color. There are two cases.

    \textbf{Case 1.} $L_{f'}(e_4)\cap L_{f'}(e_5)\neq\emptyset$.

    In this case, first, we use one color in $L_{f'}(e_4)\cap L_{f'}(e_5)$ on both $e_4$ and $e_5$. Then, there are at least three colors available for $e_1$, at least three colors available for $e_2$, and at least three colors available for $e_3$. So, we can find feasible colors for $e_1$, $e_2$, and $e_3$, and extend $f'$ to a strong $13$-edge-coloring of $H$, a contradiction.

    \textbf{Case 2.} $L_{f'}(e_4)\cap L_{f'}(e_5)=\emptyset$.

    In this case, we invoke Hall's marriage theorem (Theorem \ref{hall}). Let $S_1=L_{f'}(e_1)$, $S_2=L_{f'}(e_2)$, $S_3=L_{f'}(e_3)$, $S_4=L_{f'}(e_4)$, and $S_5=L_{f'}(e_5)$. It is clear that, for any $k\in [1,4]$ and any $1\le i_1\neq i_2\neq\dots\neq i_k\le 5$, we have $|\bigcup_{j=1}^k S_{i_j}|\ge k$. Then, for $k=5$, we have $|S_1\cup S_2\cup S_3\cup S_4\cup S_5|\ge |S_4\cup S_5|=|L_{f'}(e_4)\cup L_{f'}(e_5)|\ge 5$, as $L_{f'}(e_4)$ and $L_{f'}(e_5)$ are disjoint. Thus, we can find a system of distinct representatives of $\{S_1,S_2,S_3,S_4,S_5\}$, which means we can find feasible colors for $e_1$, $e_2$, $e_3$, $e_4$, and $e_5$ to extend $f'$ to a strong $13$-edge-coloring of $H$, a contradiction.

    In either case, we get a contradiction, so this claim is proved.
\end{proof}

In the minimal counterexample $H$, for $d\in\{2,3,4\}$, let each $d$-vertex $v$ have initial charge $\omega(v)=d-\frac{34}{11}$. By the assumption $mad(H)<\frac{34}{11}$, we have $\sum_{v\in V(H)}\omega(v)<0$. 

If $v$ is a $d$-vertex and it is adjacent to $u$, then $v$ is called a \emph{$d$-neighbor} of $u$. Now we redistribute the charge using the following discharging rules.
\begin{enumerate}
    \item[R1.] Each $4$-vertex gives $\frac{6}{11}$ to its $2$-neighbor (if it has a $2$-neighbor); and gives $\frac{4}{33}$ to each of its $3$-neighbors.
    \item[R2.] Each $3(B)$-vertex gives $\frac{1}{22}$ to its $3$-neighbor.
    \item[R3.] By definition, a $3(C_{strong})$-vertex $x$ has two $3$-neighbors (say $y$ and $z$), and at least one of them (say $y$) is a $3(B)$-vertex. Let $x$ give $\frac{5}{132}$ to $z$.
    \item[R4.] By definition, a $3(C_{moderate})$-vertex $x$ has two $3$-neighbors (say $y$ and $z$), and at least one of them (say $y$) is a $3(C)$-vertex. Let $x$ give $\frac{1}{33}$ to $z$.
    \item[R5.] Each $3(C_{weak})$-vertex gives $\frac{1}{66}$ to each of its $3(D)$-neighbors.
\end{enumerate}

After redistributing the charge, each vertex $v$ gets a new charge $\omega'(v)$, and the total charge in $H$ does not change.

Using the claims and discharging rules we made, we can check that $\omega'(v)\ge 0$ for every $v\in V(H)$.

For the $4$-vertices in $H$, by Claim \ref{4v} and R1, we have:
\begin{itemize}
    \item If $v$ is a $4$-vertex which has one $2$-neighbor and three $3$-neighbors, then $\omega'(v)=4-\frac{34}{11}-\frac{6}{11}-3\cdot\frac{4}{33}=0$.
    \item If $v$ is a $4$-vertex which has four $3$-neighbors, then $\omega'(v)=4-\frac{34}{11}-4\cdot\frac{4}{33}>0$.
\end{itemize}

For the $3(A)$-vertices, $3(B)$-vertices, and $3(C)$-vertices in $H$, we have:
\begin{itemize}
    \item If $v$ is a $3(A)$-vertex, then $\omega'(v)=3-\frac{34}{11}+3\cdot\frac{4}{33}>0$, by R1.
    \item If $v$ is a $3(B)$-vertex, then $\omega'(v)\ge 3-\frac{34}{11}+2\cdot\frac{4}{33}-\frac{1}{22}>0$, by R1 and R2.
    \item If $v$ is a $3(C_{strong})$-vertex, then $\omega'(v)\ge 3-\frac{34}{11}+\frac{4}{33}+\frac{1}{22}-\frac{5}{132}>0$, by R1, R2, and R3.
    \item If $v$ is a $3(C_{moderate})$-vertex, then $\omega'(v)\ge 3-\frac{34}{11}+\frac{4}{33}-\frac{1}{33}=0$, by R1 and R4.
    \item If $v$ is a $3(C_{weak})$-vertex, then $\omega'(v)=3-\frac{34}{11}+\frac{4}{33}-2\cdot\frac{1}{66}=0$, by R1 and R5.
\end{itemize}

Then, we know that only $3(B)$-vertices, $3(C_{strong})$-vertices, $3(C_{moderate})$-vertices, $3(C_{weak})$-vertices, and $3(D)$-vertices can be the neighbors of a $3(D)$-vertex. Let us compare how much charge they can give to a $3(D)$-vertex.

Let $x$ be a $3(D)$-vertex.
\begin{itemize}
    \item By R2, a $3(B)$-neighbor of $x$ gives $\frac{1}{22}$ to $x$.
    \item By R3, a $3(C_{strong})$-neighbor of $x$ gives $\frac{5}{132}$ to $x$.
    \item By R4, a $3(C_{moderate})$-neighbor of $x$ gives $\frac{1}{33}$ to $x$.
    \item By R5, a $3(C_{weak})$-neighbor of $x$ gives $\frac{1}{66}$ to $x$.
    \item According to our discharging rules, a $3(D)$-neighbor of $x$ does not give any charge to $x$.
\end{itemize}

We have
\begin{align*}
    \frac{1}{22}>\frac{5}{132}>\frac{1}{33}>\frac{1}{66},
\end{align*}
so, based on the amount of charge that a $3(D)$-vertex can get, we have the following order.
\begin{align}\label{order}
    3(B)>3(C_{strong})>3(C_{moderate})>3(C_{weak}). \tag{$\star$}
\end{align}

Now, for the $3(D)$-vertices in $H$, we have:
\begin{itemize}
    \item If $v$ is a $3(D)$-vertex which is adjacent to another $3(D)$-vertex, then, by Claim \ref{3dv}, it is adjacent to two $3(B)$-vertices. So $\omega'(v)=3-\frac{34}{11}+2\cdot\frac{1}{22}=0$, by R2.
    \item Assume that $v$ is a $3(D)$-vertex which is not adjacent to another $3(D)$-vertex, but is adjacent to at least one $3(B)$-vertex. By Claim \ref{3d1}, $v$ is adjacent to at most one $3(C_{weak})$-vertex. So, in view of \eqref{order}, at least $v$ can get $\frac{1}{22}$ from a $3(B)$-vertex, $\frac{1}{66}$ from a $3(C_{weak})$-vertex, and $\frac{1}{33}$ from a $3(C_{moderate})$-vertex. Thus, $\omega'(v)\ge 3-\frac{34}{11}+\frac{1}{22}+\frac{1}{66}+\frac{1}{33}=0$.
    \item If $v$ is a $3(D)$-vertex which is not adjacent to a $3(D)$-vertex or a $3(B)$-vertex, then it is only adjacent to $3(C)$-vertices.
    \begin{itemize}
        \item If $v$ is adjacent to a $3(C_{weak})$-vertex, then by Claim \ref{3d2}, it is adjacent to two $3(C_{strong})$-vertices. So $\omega'(v)=3-\frac{34}{11}+\frac{1}{66}+2\cdot\frac{5}{132}=0$, by R3 and R5.
        \item If $v$ is not adjacent to a $3(C_{weak})$-vertex, then it is only adjacent to $3(C_{strong})$-vertices and $3(C_{moderate})$-vertices. So $\omega'(v)\ge 3-\frac{34}{11}+3\cdot\frac{1}{33}=0$, by \eqref{order} and R4.
    \end{itemize}
\end{itemize}

Finally, for the $2$-vertices in $H$, by Claim \ref{2v}, we know that a $2$-vertex $v$ is adjacent to two $4$-vertices. So $\omega'(v)=2-\frac{34}{11}+2\cdot\frac{6}{11}=0$.

Now, we have $\omega'(v)\ge 0$ for every vertex $v\in V(H)$. Since the total charge in $H$ does not change, we have
\begin{align*}
    0\le \sum_{v\in V(H)}\omega'(v)=\sum_{v\in V(H)}\omega(v)<0,
\end{align*}
a contradiction. Hence, a minimal counterexample $H$ does not exist, and Theorem \ref{thm1} is established.

\section{Proof of Theorem \ref{thm2}}
In this section, we study the graphs with Ore-degree $8$ and prove Theorem \ref{thm2}.

Suppose to the contrary that Theorem \ref{thm2} is not true, and $H$ is a minimal counterexample to Theorem \ref{thm2}, which means $\theta(H)\le 8$, $mad(H)<\frac{113}{31}$, and $\chi_s'(H)>20$, but for any $v\in V(H)$, we have $\chi_s'(H-v)\le 20$.

A strong $20$-edge-coloring $f:E(H-v)\longrightarrow \{1,2,\dots,20\}$ of $H-v$ is called a \emph{partial coloring} of $H$. For an edge $e\in E(H)\setminus E(H-v)$, if a color $c\in \{1,2,\dots,20\}$ has not been used on an edge in $H-v$ that $e$ sees, then on the basis of $f$, we can color $e$ by $c$. Let $L_f(e)\subseteq \{1,2,\dots,20\}$ be the set of colors that have not been used on an edge that $e$ sees, and let $l_f(e):=|L_f(e)|$. It is clear that if $e$ sees $k\le 19$ edges in $H-v$, then $l_f(e)\ge 20-k\ge 1$.

In this proof, the following lemma plays an analogous role to Lemma \ref{lemma} in the proof of Theorem \ref{thm1}.

\begin{lemma}\label{lemma2}
    Let $v$ be a vertex in $H$ with $k$ edges $e_1,e_2,\dots,e_k$ incident to it. Let $f:E(H-v)\longrightarrow \{1,2,\dots,20\}$ be a strong $20$-edge-coloring of $H-v$, which is a partial coloring of $H$. If, up to re-ordering $e_1,e_2,\dots,e_k$, we have 
    \begin{align*}
        l_f(e_i)\ge i
    \end{align*}
    for each $i\in [1,k]$, then $f$ can be extended to a strong $20$-edge-coloring of $H$.
\end{lemma}

Using this lemma, we prove some conclusions about the structure of $H$. As $\theta(H)\le 8$, we have $d(v)\le 7$ for any $v\in V(H)$.

\begin{claim}
    There are no $1$-vertices, $2$-vertices, $6$-vertices, or $7$-vertices in $H$.
\end{claim}

This claim tells us that only $3$-vertices, $4$-vertices, and $5$-vertices can be in $H$.

\begin{proof}
    Suppose that $x\in V(H)$ is a $1$-vertex, and its only neighbor is $y$. Let $d:=d(y)\in [1,7]$. By the minimality of $H$, we know that $H-x$ has a strong $20$-edge-coloring $f$. But now, $xy$ sees at most $(d-1)(8-d)\le 12$ edges in $H-x$, which means $l_f(xy)\ge 8$. So, by Lemma \ref{lemma2}, we can extend $f$ to a strong $20$-edge-coloring of $H$, a contradiction. Thus, there are no $1$-vertices in $H$. As $7$-vertices can only be adjacent to $1$-vertices, we also know that there are no $7$-vertices in $H$.

    Suppose that $u\in V(H)$ is a $2$-vertex, and it is adjacent to $v_1$ and $v_2$. Let $d_1:=d(v_1)\in [2,6]$ and $d_2:=d(v_2)\in [2,6]$. By the minimality of $H$, we know that $H-u$ has a strong $20$-edge-coloring $f$. But now, $uv_1$ sees at most $(d_1-1)(8-d_1)+d_2-1\le 17$ edges in $H-u$, which means $l_f(uv_1)\ge 3$; and $uv_2$ sees at most $(d_2-1)(8-d_2)+d_1-1\le 17$ edges in $H-u$, which means $l_f(uv_2)\ge 3$. So, by Lemma \ref{lemma2}, we can extend $f$ to a strong $20$-edge-coloring of $H$, a contradiction. Thus, there are no $2$-vertices in $H$. As $6$-vertices can only be adjacent to $1$-vertices and $2$-vertices, we also know that there are no $6$-vertices in $H$.
\end{proof}

\begin{claim}\label{83v}
    If a $3$-vertex in $H$ is adjacent to another $3$-vertex, then it is adjacent to two $5$-vertices.
\end{claim}

This claim also implies that a $3$-vertex in $H$ is adjacent to at most one $3$-vertex.

\begin{proof}
    Let $x$ be a $3$-vertex in $H$ which is adjacent to a $3$-vertex $y$ and two other vertices $z$ and $w$. Suppose that at most one of $z$ and $w$ is a $5$-vertex.

    By symmetry, without loss of generality, we may assume that $z$ is not a $5$-vertex, which means $d_1:=d(z)\in [3,4]$; and $w$ can be a $5$-vertex, which means $d_2:=d(w)\in [3,5]$.

    By the minimality of $H$, we know that $H-x$ has a strong $20$-edge-coloring $f$. But now, $xy$ sees at most $10+(d_1-1)+(d_2-1)\le 17$ edges in $H-x$, which means $l_f(xy)\ge 3$; $xz$ sees at most $(d_1-1)(8-d_1)+2+(d_2-1)\le 18$ edges in $H-x$, which means $l_f(xz)\ge 2$; and $xw$ sees at most $(d_2-1)(8-d_2)+2+(d_1-1)\le 17$ edges in $H-x$, which means $l_f(xz)\ge 3$. So, by Lemma \ref{lemma2}, we can extend $f$ to a strong $20$-edge-coloring of $H$, a contradiction. Thus, both $z$ and $w$ must be $5$-vertices.
\end{proof}

We divide the $3$-vertices in $H$ into four classes.
\begin{itemize}
    \item A $3$-vertex is called a $3(A)$-vertex if it is adjacent to three $5$-vertices;
    \item A $3$-vertex is called a $3(B)$-vertex if it is adjacent to exactly two $5$-vertices;
    \item A $3$-vertex is called a $3(C)$-vertex if it is adjacent to exactly one $5$-vertex;
    \item A $3$-vertex is called a $3(D)$-vertex if it is not adjacent to a $5$-vertex.
\end{itemize}

A $3(B)$-vertex is either adjacent to a $3$-vertex or adjacent to a $4$-vertex. We further divide the $3(B)$-vertices into two subclasses.
\begin{itemize}
    \item A $3(B)$-vertex is called a $3(B_{strong})$-vertex if it is adjacent to two $5$-vertices and one $4$-vertex;
    \item A $3(B)$-vertex is called a $3(B_{weak})$-vertex if it is adjacent to two $5$-vertices and one $3$-vertex.
\end{itemize}

A $3(C)$-vertex must be adjacent to one $5$-vertex and two $4$-vertices, because if it is adjacent to a $3$-vertex, then by Claim \ref{83v}, it is adjacent to two $5$-vertices, contradicting the definition of a $3(C)$-vertex.

Similarly, a $3(D)$-vertex must be adjacent to three $4$-vertices, because if it is adjacent to a $3$-vertex, then by Claim \ref{83v}, it is adjacent to two $5$-vertices, contradicting the definition of a $3(D)$-vertex. In Claim \ref{83dv}, we will prove a further result about the $3(D)$-vertices, for which we need to divide the $4$-vertices into different classes. 

Before dividing the $4$-vertices into classes, we prove the following claim.

\begin{claim}\label{84v}
    A $4$-vertex in $H$ cannot be adjacent to four $3$-vertices.
\end{claim}

\begin{proof}
    Suppose to the contrary that $x$ is a $4$-vertex in $H$, and it is adjacent to four $3$-vertices $y_1$, $y_2$, $y_3$, and $y_4$.

    By the minimality of $H$, we know that $H-x$ has a strong $20$-edge-coloring $f$. But now, it is easy to check that each one of $xy_1$, $xy_2$, $xy_3$, and $xy_4$ sees at most $16$ edges in $H-x$, which means $l_f(xy_i)\ge 4$ for each $i\in [1,4]$. So, by Lemma \ref{lemma2}, we can extend $f$ to a strong $20$-edge-coloring of $H$, a contradiction. Thus, a $4$-vertex in $H$ cannot be adjacent to four $3$-vertices.
\end{proof}

Hence, we can divide the $4$-vertices into four classes.
\begin{itemize}
    \item A $4$-vertex is called a $4(A)$-vertex if it is adjacent to four $4$-vertices;
    \item A $4$-vertex is called a $4(B)$-vertex if it is adjacent to three $4$-vertices and one $3$-vertex;
    \item A $4$-vertex is called a $4(C)$-vertex if it is adjacent to two $4$-vertices and two $3$-vertices;
    \item A $4$-vertex is called a $4(D)$-vertex if it is adjacent to one $4$-vertex and three $3$-vertices.
\end{itemize}

\begin{claim}\label{83dv}
    A $3(D)$-vertex in $H$ is adjacent to three $4(B)$-vertices.
\end{claim}

\begin{proof}
    We have mentioned that a $3(D)$-vertex must be adjacent to three $4$-vertices. Also, it is clear that a $3(D)$-vertex can only be adjacent to $4(B)$-vertices, $4(C)$-vertices, and $4(D)$-vertices. 
    
    Let $x$ be a $3(D)$-vertex which is adjacent to three $4$-vertices $y_1$, $y_2$, and $y_3$. Suppose to the contrary that $x$ is not adjacent to three $4(B)$-vertices. By symmetry, without loss of generality, we may assume that $y_1$ is not a $4(B)$-vertex. So $y_1$ is adjacent to at least two $3$-vertices, one of which is $x$.

    By the minimality of $H$, we know that $H-x$ has a strong $20$-edge-coloring $f$. But now, it is easy to check that $xy_1$ sees at most $17$ edges in $H-x$, which means $l_f(xy_1)\ge 3$; $xy_2$ sees at most $18$ edges in $H-x$, which means $l_f(xy_2)\ge 2$; and $xy_3$ sees at most $18$ edges in $H-x$, which means $l_f(xy_3)\ge 2$. So, by Lemma \ref{lemma2}, we can extend $f$ to a strong $20$-edge-coloring of $H$, a contradiction. Thus, a $3(D)$-vertex in $H$ must be adjacent to three $4(B)$-vertices.
\end{proof}

Now we analyze the neighbors of the $4$-vertices in $H$.

\begin{claim}\label{84dv}
    A $4(D)$-vertex is adjacent to three $3(B_{strong})$-vertices and one $4(A)$-vertex or $4(B)$-vertex.
\end{claim}

\begin{proof}
    Let $x$ be a $4(D)$-vertex, which is adjacent to three $3$-vertices $y_1$, $y_2$, and $y_3$, and a $4$-vertex $z$.
    
    First, we prove that $y_1$, $y_2$, and $y_3$ are all $3(B_{strong})$-vertices. 
    
    As $y_1$, $y_2$, and $y_3$ are already adjacent to a $4$-vertex $x$, by the classification of the $3$-vertices, we know that none of $y_1$, $y_2$, and $y_3$ can be a $3(A)$-vertex. Also, by Claim \ref{83dv}, we know that none of $y_1$, $y_2$, and $y_3$ can be a $3(D)$-vertex. Now we prove that none of $y_1$, $y_2$, and $y_3$ can be a $3(C)$-vertex.
    
    Suppose to the contrary that at least one of $y_1$, $y_2$, and $y_3$ is a $3(C)$-vertex. By symmetry, we may assume that $y_1$ is a $3(C)$-vertex. Then, by the minimality of $H$, we know that $H-x$ has a strong $20$-edge-coloring $f$. But now, it is easy to check that $xy_1$ sees at most $16$ edges in $H-x$, which means $l_f(xy_1)\ge 4$; $xy_2$ sees at most $17$ edges in $H-x$, which means $l_f(xy_2)\ge 3$; $xy_3$ sees at most $17$ edges in $H-x$, which means $l_f(xy_3)\ge 3$; and $xz$ sees at most $18$ edges in $H-x$, which means $l_f(xz)\ge 2$. So, by Lemma \ref{lemma2}, we can extend $f$ to a strong $20$-edge-coloring of $H$, a contradiction. Thus, none of $y_1$, $y_2$, and $y_3$ can be a $3(C)$-vertex, and this means all three of them are $3(B)$-vertices. 
    
    Recall that a $3(B_{strong})$-vertex is adjacent to two $5$-vertices and one $4$-vertex; and a $3(B_{weak})$-vertex is adjacent to two $5$-vertices and one $3$-vertex. As $y_1$, $y_2$, and $y_3$ are adjacent to a $4$-vertex $x$, we know that they must be $3(B_{strong})$-vertices.

    Then, we prove that $z$ is a $4(A)$-vertex or a $4(B)$-vertex.

    Suppose to the contrary that $z$ is a $4(C)$-vertex or a $4(D)$-vertex. Then, by the minimality of $H$, we know that $H-x$ has a strong $20$-edge-coloring $f$. But now, it is easy to check that $xy_1$ sees at most $17$ edges in $H-x$, which means $l_f(xy_1)\ge 3$; $xy_2$ sees at most $17$ edges in $H-x$, which means $l_f(xy_2)\ge 3$; $xy_3$ sees at most $17$ edges in $H-x$, which means $l_f(xy_3)\ge 3$; and $xz$ sees at most $16$ edges in $H-x$, which means $l_f(xz)\ge 4$. So, by Lemma \ref{lemma2}, we can extend $f$ to a strong $20$-edge-coloring of $H$, a contradiction. Thus, $z$ must be a $4(A)$-vertex or a $4(B)$-vertex.
\end{proof}

By Claim \ref{83dv}, $3(D)$-vertices are only adjacent to $4(B)$-vertices, so $4(C)$-vertices cannot be adjacent to $3(D)$-vertices. Then, we can further divide the $4(C)$-vertices in $H$ into two subclasses.

\begin{itemize}
    \item A $4(C)$-vertex is called a $4(C_{strong})$-vertex if it is adjacent to at most one $3(C)$-vertex;
    \item A $4(C)$-vertex is called a $4(C_{weak})$-vertex if it is adjacent to two $3(C)$-vertices.
\end{itemize}

\begin{claim}\label{4cstrong}
    A $4(C_{strong})$-vertex is either adjacent to one $3(C)$-vertex and one $3(B_{strong})$-vertex; or adjacent to two $3(B_{strong})$-vertices.
\end{claim}

\begin{proof}
    Let $x$ be a $4(C_{strong})$-vertex. We know that $x$ is not adjacent to $3(A)$-vertices or $3(D)$-vertices, and is adjacent to at most one $3(C)$-vertex.

    If $x$ is adjacent to a $3(C)$-vertex, then the other $3$-vertex adjacent to $x$ must be a $3(B)$-vertex. Let this $3(B)$-vertex be $y$. As $y$ is adjacent to a $4$-vertex $x$, by the subclassification of $3(B)$-vertices, we know that $y$ is a $3(B_{strong})$-vertex.

    If $x$ is not adjacent to a $3(C)$-vertex, then both of the $3$-vertices adjacent to $x$ must be $3(B)$-vertices. As these two $3(B)$-vertices are both adjacent to a $4$-vertex $x$, by the subclassification of $3(B)$-vertices, we know that they are $3(B_{strong})$-vertices.
\end{proof}

In Claim \ref{84cv}, we will analyze the neighbors of the $4(C_{weak})$-vertices. To establish Claim \ref{84cv}, we need the conclusions of Claims \ref{4ctriangle} and \ref{34c4c4c}.

\begin{claim}\label{4ctriangle}
    A triangle, whose three vertices are all $4(C)$-vertices, cannot be a subgraph of $H$.
\end{claim}

\begin{proof}
    Suppose to the contrary that $H$ has a triangle subgraph whose three vertices $x_1$, $x_2$, and $x_3$ are $4(C)$-vertices. Let $x_1 y_1$ and $x_1 y_2$ be the other two edges incident to $x_1$, besides $x_1 x_2$ and $x_1 x_3$. As $x_1$ is a $4(C)$-vertex, we know that both $y_1$ and $y_2$ are $3$-vertices.

    By the minimality of $H$, we know that $H-x_1$ has a strong $20$-edge-coloring $f$. But now, it is easy to check that $x_1 x_2$ sees at most $13$ edges in $H-x_1$, which means $l_f(x_1 x_2)\ge 7$; $x_1 x_3$ sees at most $13$ edges in $H-x_1$, which means $l_f(x_1 x_3)\ge 7$; $x_1 y_1$ sees at most $17$ edges in $H-x_1$, which means $l_f(x_1 y_1)\ge 3$; and $x_1 y_2$ sees at most $17$ edges in $H-x_1$, which means $l_f(x_1 y_2)\ge 3$. So, by Lemma \ref{lemma2}, we can extend $f$ to a strong $20$-edge-coloring of $H$, a contradiction. Thus, a triangle, whose three vertices are all $4(C)$-vertices, cannot be a subgraph of $H$.
\end{proof}

\begin{claim}\label{3triangle}
    A triangle, which has at least one $3$-vertex, cannot be a subgraph of $H$.
\end{claim}

\begin{proof}
    This proof is very similar to the proof of Claim \ref{4ctriangle}, so we omit it here.
\end{proof}

\begin{claim}\label{34c4c4c}
    A $4$-cycle $x_1\sim x_2\sim x_3\sim x_4\sim x_1$, where $x_1$ is a $3$-vertex, and $x_2$, $x_3$, and $x_4$ are $4(C)$-vertices, cannot be a subgraph of $H$.
\end{claim}

\begin{proof}
    Suppose to the contrary that there is such a $4$-cycle in $H$. Let $x_1 y$ be the third edge incident to $x_1$ (the other two are $x_1 x_2$ and $x_1 x_4$). Note that $y\neq x_3$, because otherwise we have a triangle with a $3$-vertex, which is forbidden by Claim \ref{3triangle}. By the minimality of $H$, we know that $H-x_1$ has a strong $20$-edge-coloring $f$. But now, it is easy to check that $x_1 y$ sees at most $18$ edges in $H-x_1$, which means $l_f(x_1 y)\ge 2$; $x_1 x_2$ sees at most $17$ edges in $H-x_1$, which means $l_f(x_1 x_2)\ge 3$; and $x_1 x_4$ sees at most $17$ edges in $H-x_1$, which means $l_f(x_1 x_4)\ge 3$. Thus, by Lemma \ref{lemma2}, we can extend $f$ to a strong $20$-edge-coloring of $H$, a contradiction.
\end{proof}

By Claim \ref{84dv}, we know that a $4(C)$-vertex cannot be adjacent to $4(D)$-vertices. In the following claim, we further prove that a $4(C_{weak})$-vertex is adjacent to at least one $4(A)$-vertex or $4(B)$-vertex --- in other words, a $4(C_{weak})$-vertex cannot be adjacent to two $4(C)$-vertices.

\begin{claim}\label{84cv}
    A $4(C_{weak})$-vertex is adjacent to at least one $4(A)$-vertex or $4(B)$-vertex.
\end{claim}

\begin{proof}
    Let $x$ be a $4(C_{weak})$-vertex. By the definition of a $4(C_{weak})$-vertex, we know that $x$ is adjacent to two $3(C)$-vertices, say $y_1$ and $y_2$. Now, suppose to the contrary that $x$ is adjacent to two $4(C)$-vertices, say $z_1$ and $z_2$. Let $w_1$ be a $3$-vertex adjacent to $z_1$, and let $w_2$ be a $3$-vertex adjacent to $z_2$.

    For convenience, denote $xy_1$ by $e_1$; $xy_2$ by $e_2$; $xz_1$ by $e_3$; $xz_2$ by $e_4$; $z_1 w_1$ by $e_5$; and $z_2 w_2$ by $e_6$.

    We know that $e_5$ and $e_6$ are not the same edge. Because otherwise we have $z_1=w_2$ and $z_2=w_1$, which is impossible because $z_1$ and $z_2$ are $4(C)$-vertices and $w_1$ and $w_2$ are $3$-vertices.
    
    Then, we check that $e_5$ and $e_6$ do not see each other.

    If $e_5$ and $e_6$ are incident, then there are three possible cases.
    \begin{enumerate}
        \item $z_1=w_2$. This is not possible because $z_1$ is a $4(C)$-vertex and $w_2$ is a $3$-vertex.
        \item $z_2=w_1$. This is not possible because $z_2$ is a $4(C)$-vertex and $w_1$ is a $3$-vertex.
        \item $w_1=w_2$. In this case, $w_1\sim z_1\sim x\sim z_2\sim w_1$ is a $4$-cycle forbidden by Claim \ref{34c4c4c}, a contradiction.
    \end{enumerate}

    If $e_5$ and $e_6$ are not incident, but are joined by another edge, then there are four possible cases.
    \begin{enumerate}
        \item $z_1$ and $z_2$ are joined by an edge. In this case, $x\sim z_1\sim z_2\sim x$ is a triangle of $4(C)$-vertices, which is forbidden by Claim \ref{4ctriangle}, a contradiction.
        \item $z_1$ and $w_2$ are joined by an edge. In this case, $w_2\sim z_1\sim x\sim z_2\sim w_2$ is a $4$-cycle forbidden by Claim \ref{34c4c4c}, a contradiction.
        \item $z_2$ and $w_1$ are joined by an edge. In this case, $w_1\sim z_2\sim x\sim z_1\sim w_1$ is a $4$-cycle forbidden by Claim \ref{34c4c4c}, a contradiction.
        \item $w_1$ and $w_2$ are joined by an edge. In this case, two $3$-vertices $w_1$ and $w_2$ are adjacent to each other, so by Claim \ref{83v}, all their neighbors should be $5$-vertices. However, $w_1$ is adjacent to a $4(C)$-vertex $z_1$, and $w_2$ is adjacent to a $4(C)$-vertex $z_2$, a contradiction.
    \end{enumerate}

    Thus, $e_5$ and $e_6$ do not see each other.

    By the minimality of $H$, we know that $H-x$ has a strong $20$-edge-coloring $f$. It is easy to check that, for each $i\in [1,4]$, $e_i$ sees at most $17$ edges (including $e_5$ and $e_6$) in $H-x$, which means $l_f(e_i)\ge 3$. 
    
    In $H-x$, $e_5$ sees at most $17$ edges, and $e_6$ also sees at most $17$ edges. Now we erase the colors on $e_5$ and $e_6$ under $f$ to get a new partial coloring $f'$. Under $f'$, for each $i\in [1,4]$, $e_i$ sees at most $15$ edges with colors, which means $l_{f'}(e_i)\ge 5$. Also, under $f'$, $e_5$ sees at most $17$ edges with colors, which means $l_{f'}(e_5)\ge 3$; and $e_6$ sees at most $17$ edges with colors, which means $l_{f'}(e_6)\ge 3$.
    
    As $e_5$ does not see $e_6$, we may color them by the same color. There are two cases.

    \textbf{Case 1.} $L_{f'}(e_5)\cap L_{f'}(e_6)\neq\emptyset$.

    In this case, first, we use one color in $L_{f'}(e_5)\cap L_{f'}(e_6)$ on both $e_5$ and $e_6$. Then, there are at least four colors available for each $e_i$ with $i\in [1,4]$. So, we can find feasible colors for $e_1$, $e_2$, $e_3$, and $e_4$, and extend $f'$ to a strong $20$-edge-coloring of $H$, a contradiction.

    \textbf{Case 2.} $L_{f'}(e_5)\cap L_{f'}(e_6)=\emptyset$.

    In this case, we invoke Hall's marriage theorem (Theorem \ref{hall}). For $k\in [1,6]$, let $S_k=L_{f'}(e_k)$. It is clear that, for any $k\in [1,5]$ and any $1\le i_1\neq i_2\neq\dots\neq i_k\le 6$, we have $|\bigcup_{j=1}^k S_{i_j}|\ge k$. Then, for $k=6$, we have $|S_1\cup S_2\cup S_3\cup S_4\cup S_5\cup S_6|\ge |S_5\cup S_6|=|L_{f'}(e_5)\cup L_{f'}(e_6)|\ge 6$, as $L_{f'}(e_5)$ and $L_{f'}(e_6)$ are disjoint. Thus, we can find a system of distinct representatives of $\{S_1,S_2,S_3,S_4,S_5,S_6\}$, which means we can find feasible colors for $e_1$, $e_2$, $e_3$, $e_4$, $e_5$, and $e_6$ to extend $f'$ to a strong $20$-edge-coloring of $H$, a contradiction.

    In either case, we get a contradiction, so a $4(C_{weak})$-vertex is adjacent to at least one $4(A)$-vertex or $4(B)$-vertex.
\end{proof}

Then, regarding the $5$-vertices in $H$, we know that a $5$-vertex must be adjacent to five $3$-vertices, as $\theta(H)\le 8$. Now we prove that at most one of them can be a $3(B_{weak})$-vertex.

\begin{claim}\label{85v}
    A $5$-vertex in $H$ is adjacent to at most one $3(B_{weak})$-vertex.
\end{claim}

\begin{proof}
    Let $x$ be a $5$-vertex in $H$, which is adjacent to five $3$-vertices $y_1$, $y_2$, $y_3$, $y_4$, and $y_5$. Suppose to the contrary that at least two of them, say $y_1$ and $y_2$, are $3(B_{weak})$-vertices. By the definition of $3(B_{weak})$-vertices, we know that each of $y_1$ and $y_2$ is adjacent to a $3$-vertex. Let $z_1$ be the $3$-vertex adjacent to $y_1$, and let $z_2$ be the $3$-vertex adjacent to $y_2$.

    For convenience, denote $xy_1$ by $e_1$; $xy_2$ by $e_2$; $xy_3$ by $e_3$; $xy_4$ by $e_4$; $xy_5$ by $e_5$; $y_1 z_1$ by $e_6$; and $y_2 z_2$ by $e_7$.

    We know that $e_6$ and $e_7$ are not the same edge. Because otherwise we have $y_1=z_2$ and $y_2=z_1$, so $x\sim y_1\sim y_2\sim x$ is a triangle having a $3$-vertex, which is forbidden by Claim \ref{3triangle}, a contradiction.
    
    Then, we check that $e_6$ and $e_7$ do not see each other.

    If $e_6$ and $e_7$ are incident, then there are three possible cases.
    \begin{enumerate}
        \item $y_1=z_2$. In this case, $x\sim y_1\sim y_2\sim x$ is a triangle forbidden by Claim \ref{3triangle}, a contradiction.
        \item $y_2=z_1$. In this case, $x\sim y_2\sim y_1\sim x$ is a triangle forbidden by Claim \ref{3triangle}, a contradiction.
        \item $z_1=z_2$. In this case, two $3$-vertices $y_1$ and $z_1$ are adjacent to each other, so by Claim \ref{83v}, all their neighbors should be $5$-vertices. However, $z_1$ is adjacent to another $3$-vertex $y_2$, a contradiction.
    \end{enumerate}

    If $e_6$ and $e_7$ are not incident, but are joined by another edge, then there are four possible cases.
    \begin{enumerate}
        \item $y_1$ and $y_2$ are joined by an edge. In this case, $x\sim y_1\sim y_2\sim x$ is a triangle forbidden by Claim \ref{3triangle}, a contradiction.
        \item $y_1$ and $z_2$ are joined by an edge. In this case, two $3$-vertices $y_1$ and $z_2$ are adjacent to each other, so by Claim \ref{83v}, all their neighbors should be $5$-vertices. However, $z_2$ is adjacent to another $3$-vertex $y_2$, a contradiction.
        \item $y_2$ and $z_1$ are joined by an edge. In this case, two $3$-vertices $y_2$ and $z_1$ are adjacent to each other, so by Claim \ref{83v}, all their neighbors should be $5$-vertices. However, $z_1$ is adjacent to another $3$-vertex $y_1$, a contradiction.
        \item $z_1$ and $z_2$ are joined by an edge. In this case, two $3$-vertices $z_1$ and $z_2$ are adjacent to each other, so by Claim \ref{83v}, all their neighbors should be $5$-vertices. However, $z_1$ is adjacent to another $3$-vertex $y_1$, and $z_2$ is adjacent to another $3$-vertex $y_2$, a contradiction.
    \end{enumerate}

    Thus, $e_6$ and $e_7$ do not see each other.

    By the minimality of $H$, we know that $H-x$ has a strong $20$-edge-coloring $f$. It is easy to check that, for each $i\in \{1,2\}$, $e_i$ sees at most $16$ edges (including $e_6$ and $e_7$) in $H-x$, which means $l_f(e_i)\ge 4$. Also, for each $j\in \{3,4,5\}$, $e_j$ sees at most $18$ edges (including $e_6$ and $e_7$) in $H-x$, which means $l_f(e_j)\ge 2$.
    
    In $H-x$, $e_6$ sees at most $15$ edges, and $e_7$ also sees at most $15$ edges. Now we erase the colors on $e_6$ and $e_7$ under $f$ to get a new partial coloring $f'$. Under $f'$, for each $i\in \{1,2\}$, $e_i$ sees at most $14$ edges with colors, which means $l_{f'}(e_i)\ge 6$. For each $j\in \{3,4,5\}$, $e_j$ sees at most $16$ edges with colors, which means $l_{f'}(e_j)\ge 4$. Also, under $f'$, $e_6$ sees at most $15$ edges with colors, which means $l_{f'}(e_6)\ge 5$; and $e_7$ sees at most $15$ edges with colors, which means $l_{f'}(e_7)\ge 5$.

    As $e_6$ does not see $e_7$, we may color them by the same color. There are two cases.

    \textbf{Case 1.} $L_{f'}(e_6)\cap L_{f'}(e_7)\neq\emptyset$.

    In this case, first, we use one color in $L_{f'}(e_6)\cap L_{f'}(e_7)$ on both $e_6$ and $e_7$. Then, there are at least five colors available for each $e_i$ with $i\in \{1,2\}$; and at least three colors available for each $e_j$ with $j\in \{3,4,5\}$. So, we can find feasible colors for $e_1$, $e_2$, $e_3$, $e_4$, and $e_5$, and extend $f'$ to a strong $20$-edge-coloring of $H$, a contradiction.

    \textbf{Case 2.} $L_{f'}(e_6)\cap L_{f'}(e_7)=\emptyset$.

    In this case, we invoke Hall's marriage theorem (Theorem \ref{hall}). For $k\in [1,7]$, let $S_k=L_{f'}(e_k)$. It is clear that, for any $k\in [1,6]$ and any $1\le i_1\neq i_2\neq\dots\neq i_k\le 7$, we have $|\bigcup_{j=1}^k S_{i_j}|\ge k$. Then, for $k=7$, we have $|S_1\cup S_2\cup S_3\cup S_4\cup S_5\cup S_6\cup S_7|\ge |S_6\cup S_7|=|L_{f'}(e_6)\cup L_{f'}(e_7)|\ge 10>7$, as $L_{f'}(e_6)$ and $L_{f'}(e_7)$ are disjoint. Thus, we can find a system of distinct representatives of $\{S_1,S_2,S_3,S_4,S_5,S_6,S_7\}$, which means we can find feasible colors for $e_1$, $e_2$, $e_3$, $e_4$, $e_5$, $e_6$, and $e_7$ to extend $f'$ to a strong $20$-edge-coloring of $H$, a contradiction.

    In either case, we get a contradiction, so a $5$-vertex in $H$ is adjacent to at most one $3(B_{weak})$-vertex.
\end{proof}

In the minimal counterexample $H$, for $d\in\{3,4,5\}$, let each $d$-vertex $v$ have initial charge $\omega(v)=d-\frac{113}{31}$. By the assumption $mad(H)<\frac{113}{31}$, we have $\sum_{v\in V(H)}\omega(v)<0$. 

If $v$ is a $d$-vertex and it is adjacent to $u$, then $v$ is called a \emph{$d$-neighbor} of $u$. Now we redistribute the charge using the following discharging rules.

\begin{enumerate}
    \item[R1.] Each $5$-vertex gives $\frac{10}{31}$ to its $3(B_{weak})$-neighbor (if it has a $3(B_{weak})$-neighbor); and gives $\frac{8}{31}$ to each $3$-neighbor that is not a $3(B_{weak})$-vertex.
    \item[R2.] Each $4(A)$-vertex gives $\frac{11}{124}$ to each of its $4$-neighbors.
    \item[R3.] Each $4(B)$-vertex gives $\frac{8}{31}$ to its $3$-neighbor; and gives $\frac{1}{31}$ to each of its $4$-neighbors.
    \item[R4.] Each $4(C_{strong})$-vertex gives $\frac{7}{31}$ to its $3(C)$-neighbor (if it has a $3(C)$-neighbor); and gives $\frac{4}{31}$ to each of its $3(B_{strong})$-neighbors.
    \item[R5.] Each $4(C_{weak})$-vertex gives $\frac{6}{31}$ to each of its $3(C)$-neighbors.
    \item[R6.] Each $4(D)$-vertex gives $\frac{4}{31}$ to each of its $3(B_{strong})$-neighbors.
\end{enumerate}

After redistributing the charge, each vertex $v$ gets a new charge $\omega'(v)$, and the total charge in $H$ does not change.

Using the claims and discharging rules we made, we can check that $\omega'(v)\ge 0$ for every $v\in V(H)$.

For the $5$-vertices in $H$, by Claim \ref{85v} and R1, we have:
\begin{itemize}
    \item If $v$ is a $5$-vertex which has one $3(B_{weak})$-neighbor and four $3$-neighbors that are not $3(B_{weak})$-vertices, then $\omega'(v)=5-\frac{113}{31}-\frac{10}{31}-4\cdot\frac{8}{31}=0$.
    \item If $v$ is a $5$-vertex which has five $3$-neighbors that are not $3(B_{weak})$-vertices, then $\omega'(v)=5-\frac{113}{31}-5\cdot\frac{8}{31}>0$.
\end{itemize}

Then, for the $4$-vertices in $H$, we have:
\begin{itemize}
    \item If $v$ is a $4(A)$-vertex, then $\omega'(v)=4-\frac{113}{31}-4\cdot\frac{11}{124}=0$, by R2.
    \item If $v$ is a $4(B)$-vertex, then $\omega'(v)=4-\frac{113}{31}-\frac{8}{31}-3\cdot\frac{1}{31}=0$, by R3.
    \item By Claim \ref{4cstrong}, a $4(C_{strong})$-vertex is either adjacent to one $3(C)$-vertex and one $3(B_{strong})$-vertex; or adjacent to two $3(B_{strong})$-vertices. So, if $v$ is a $4(C_{strong})$-vertex, then $\omega'(v)\ge 4-\frac{113}{31}-\frac{7}{31}-\frac{4}{31}=0$, by R4.
    \item By definition, a $4(C_{weak})$-vertex is adjacent to two $3(C)$-vertices. By Claim \ref{84cv}, a $4(C_{weak})$-vertex is adjacent to at least one $4(A)$-vertex or $4(B)$-vertex. So, if $v$ is a $4(C_{weak})$-vertex, then $\omega'(v)\ge 4-\frac{113}{31}-2\cdot\frac{6}{31}+\frac{1}{31}=0$, by R5 and R3.
    \item By Claim \ref{84dv}, a $4(D)$-vertex is adjacent to three $3(B_{strong})$-vertices and one $4(A)$-vertex or $4(B)$-vertex. So, if $v$ is a $4(D)$-vertex, then $\omega'(v)\ge 4-\frac{113}{31}-3\cdot\frac{4}{31}+\frac{1}{31}=0$, by R6 and R3.
\end{itemize}

Finally, for the $3$-vertices in $H$, we have:
\begin{itemize}
    \item If $v$ is a $3(A)$-vertex, then $\omega'(v)=3-\frac{113}{31}+3\cdot\frac{8}{31}>0$, by R1.
    \item A $3(B_{strong})$-vertex is adjacent to two $5$-vertices and one $4$-vertex. This $4$-vertex can be a $4(B)$-vertex, a $4(C_{strong})$-vertex, or a $4(D)$-vertex. So, if $v$ is a $3(B_{strong})$-vertex, then $\omega'(v)\ge 3-\frac{113}{31}+2\cdot\frac{8}{31}+\frac{4}{31}=0$, by R1 and R4 (alternatively, R1 and R6).
    \item If $v$ is a $3(B_{weak})$-vertex, then $\omega'(v)=3-\frac{113}{31}+2\cdot\frac{10}{31}=0$, by R1.
    \item As we mentioned after the classification of $3$-vertices, a $3(C)$-vertex is adjacent to one $5$-vertex and two $4$-vertices. By definition, these two $4$-vertices cannot be $4(A)$-vertices. By Claim \ref{84dv}, these two $4$-vertices cannot be $4(D)$-vertices. So, these two $4$-vertices can only be $4(B)$-vertices and/or $4(C)$-vertices. Thus, if $v$ is a $3(C)$-vertex, then $\omega'(v)\ge 3-\frac{113}{31}+\frac{8}{31}+2\cdot\frac{6}{31}=0$, by R1 and R5.
    \item By Claim \ref{83dv}, a $3(D)$-vertex is adjacent to three $4(B)$-vertices. So, if $v$ is a $3(D)$-vertex, then $\omega'(v)=3-\frac{113}{31}+3\cdot\frac{8}{31}>0$, by R3.
\end{itemize}

So, we have $\omega'(v)\ge 0$ for every vertex $v\in V(H)$. Since the total charge in $H$ does not change, we have
\begin{align*}
    0\le \sum_{v\in V(H)}\omega'(v)=\sum_{v\in V(H)}\omega(v)<0,
\end{align*}
a contradiction. Hence, a minimal counterexample $H$ does not exist, and Theorem \ref{thm2} is established.

\section{Remarks}
Given the value of $\theta(G)$, the following result from \cite{Wa} determines the largest possible value of $mad(G)$.

\begin{proposition}[Wang \cite{Wa}]
    For a graph $G$, we have
    \[
    mad(G)\le \begin{cases}
        \frac{2k(k+1)}{2k+1} & if\ \theta(G)=2k+1\ is\ odd;\\
        k & if\ \theta(G)=2k\ is\ even.
    \end{cases}
    \]
    The equality can be attained in either case.
\end{proposition}

So, if $\theta(G)=7$, then the largest possible value of $mad(G)$ is $\frac{24}{7}$, which means our bound $mad(G)<\frac{34}{11}$ in Theorem \ref{thm1} is $\frac{26}{77}$ away from the best possible bound; if $\theta(G)=8$, then the largest possible value of $mad(G)$ is $4$, which means our bound $mad(G)<\frac{113}{31}$ in Theorem \ref{thm2} is $\frac{11}{31}$ away from the best possible bound.


\begin{thebibliography}{10}

\bibitem{An}
L.~D. Andersen.
\newblock The strong chromatic index of a cubic graph is at most {$10$}.
\newblock volume 108, pages 231--252. 1992.
\newblock Topological, algebraical and combinatorial structures. Frol\'ik's memorial volume.

\bibitem{BPP}
M.~Bonamy, T.~Perrett, and L.~Postle.
\newblock Colouring graphs with sparse neighbourhoods: bounds and applications.
\newblock {\em J. Combin. Theory Ser. B}, 155:278--317, 2022.

\bibitem{BJ}
H.~Bruhn and F.~Joos.
\newblock A stronger bound for the strong chromatic index.
\newblock {\em Combin. Probab. Comput.}, 27(1):21--43, 2018.

\bibitem{CCZZ}
L.~Chen, S.~Chen, R.~Zhao, and X.~Zhou.
\newblock The strong chromatic index of graphs with edge weight eight.
\newblock {\em J. Comb. Optim.}, 40(1):227--233, 2020.

\bibitem{CHYZ}
L.~Chen, M.~Huang, G.~Yu, and X.~Zhou.
\newblock The strong edge-coloring for graphs with small edge weight.
\newblock {\em Discrete Math.}, 343(4):111779, 11, 2020.

\bibitem{EN}
P.~Erd\H{o}s and J.~Ne\v{s}et\v{r}il.
\newblock {\em Irregularities of partitions}, volume~8 of {\em Algorithms and Combinatorics: Study and Research Texts, edited by G. Hal\'asz and V.T. S\'os}.
\newblock Springer-Verlag, Berlin, 1989.
\newblock Papers from the meeting held in Fert\H od, July 7--11, 1986.

\bibitem{FJ}
J.-L. Fouquet and J.-L. Jolivet.
\newblock Strong edge-colorings of graphs and applications to multi-{$k$}-gons.
\newblock {\em Ars Combin.}, 16:141--150, 1983.

\bibitem{Ha}
P.~Hall.
\newblock On {R}epresentatives of {S}ubsets.
\newblock {\em J. London Math. Soc.}, 10(1):26--30, 1935.

\bibitem{HHT}
P.~Hor\'ak, Q.~He, and W.~T. Trotter.
\newblock Induced matchings in cubic graphs.
\newblock {\em J. Graph Theory}, 17(2):151--160, 1993.

\bibitem{HSY}
M.~Huang, M.~Santana, and G.~Yu.
\newblock Strong chromatic index of graphs with maximum degree four.
\newblock {\em Electron. J. Combin.}, 25(3):Paper No. 3.31, 24, 2018.

\bibitem{HDK}
E.~Hurley, R.~de~Joannis~de Verclos, and R.~J. Kang.
\newblock An improved procedure for colouring graphs of bounded local density.
\newblock {\em Adv. Comb.}, pages Paper No. 7, 33, 2022.

\bibitem{LL}
Y.~Lin and W.~Lin.
\newblock Strong chromatic index of claw-free graphs with edge weight seven.
\newblock {\em Discuss. Math. Graph Theory}, 44(4):1311--1325, 2024.

\bibitem{LLH}
J.~Lu, H.~Liu, and X.~Hu.
\newblock On strong edge-coloring of graphs with maximum degree 5.
\newblock {\em Discrete Appl. Math.}, 344:120--128, 2024.

\bibitem{LLY}
J.-B. Lv, X.~Li, and G.~Yu.
\newblock On strong edge-coloring of graphs with maximum degree 4.
\newblock {\em Discrete Appl. Math.}, 235:142--153, 2018.

\bibitem{MR}
M.~Molloy and B.~Reed.
\newblock A bound on the strong chromatic index of a graph.
\newblock {\em J. Combin. Theory Ser. B}, 69(2):103--109, 1997.

\bibitem{NN}
K.~Nakprasit and K.~Nakprasit.
\newblock The strong chromatic index of graphs with restricted {O}re-degrees.
\newblock {\em Ars Combin.}, 118:373--380, 2015.

\bibitem{NY}
S.~Nelson and G.~Yu.
\newblock Planar graphs with ore-degree at most seven is strongly $13$-edge-colorable.
\newblock {\em arXiv preprint arXiv:2509.06808}, 2025.

\bibitem{Wa}
R.~Wang.
\newblock Strong edge-coloring of graphs with maximum edge weight seven.
\newblock {\em J. Comb. Optim.}, 51(1):Paper No. 2, 2026.

\bibitem{WL}
J.~Wu and W.~Lin.
\newblock The strong chromatic index of a class of graphs.
\newblock {\em Discrete Math.}, 308(24):6254--6261, 2008.

\bibitem{Za}
C.~Zang.
\newblock The strong chromatic index of graphs with maximum degree \uppercase{$\Delta$}.
\newblock {\em arXiv preprint arXiv:1510.00785}, 2015.

\end{thebibliography}
\end{document}